\documentclass[a4paper,11pt,oneside]{article}
\usepackage[paperwidth=195mm,paperheight=270mm,left=20mm,right=17mm,top=20mm,bottom=20mm,includefoot]{geometry}

\usepackage{amsmath,amssymb,mathrsfs,amscd,amsthm}
\usepackage{graphicx}
\usepackage{hyperref}
\usepackage{mathtools}

\newtheorem{Theorem}{Theorem}[section]
\newtheorem{Definition}[Theorem]{Definition}
\newtheorem{Lemma}[Theorem]{Lemma}

\newtheorem{Corollary}[Theorem]{Corollary}

\begin{document}

\title{Growth Rate Gap for Stable Subgroups}
\author{Suzhen Han and Qing Liu}
\date{ }
\maketitle

\begin{abstract}
 We prove that stable subgroups of Morse local-to-global groups exhibit a growth gap. That is, the growth rate of an infinite-index stable subgroup is strictly less than the growth rate of the ambient Morse local-to-global group.
 This generalizes a result of Cordes, Russell, Spriano, and Zalloum in the sense that we removed the additional torsion-free or residually finite assumptions.
 %In our paper, the additional assumptions, virtually torsion-free and residually finite can be removed. 
 The Morse local-to-global groups are a very broad class of groups, including mapping class groups, CAT(0) groups, closed $3$-manifold groups, certain relatively hyperbolic groups, virtually solvable groups, etc.
% and any finitely generated group with an infinite order central element.

%The Morse local-to-global property is introduced by Russell, Spriano, and Tran, who also showed mapping class groups, Teichm\"{u}ller space, CAT(0) space, closed $3$-manifold groups all admits such a property. Using a direct analysis, we obtain that the growth rates of a Morse local-to-global group and its infinite-index stable subgroup have a gap, generalizing a result of Cordes, Russell, Spriano, and Zalloum that such a gap exists if moreover the group is virtually torsion-free or the stable subgroup is residually finite by considering the automatic structures on Morse geodesics.
\end{abstract}

\renewcommand{\thefootnote}{\alph{footnote}}
\setcounter{footnote}{-1} 
\renewcommand{\thefootnote}{\alph{footnote}}
\setcounter{footnote}{-1} 

\section{Introduction}
%\subsection{Background}

Generalizing the phenomena that long local quasi-geodesics are global quasi-geodesics in hyperbolic space by Gromov\cite{Gromov87}, the Morse local-to-global property for a metric space is introduced by Russell, Spriano, and Tran\cite{RST22}, which requires that long local Morse quasi-geodesics are globally Morse quasi-geodesics. They proved in the same paper that the following spaces all satisfy such property: mapping class groups, Teichm\"{u}ller space with either the Weil-Petersson or Teichm\"{u}ller metric, graph product of hyperbolic groups, virtually solvable groups, closed 3-manifold groups, etc.

Growth in groups has been extensively studied recently, and the literature contains numerous results. In this paper, we investigate the growth of stable subgroups, which are geometrically well-behaved in groups that exhibit the Morse local-to-global property.

Suppose that a group $G$ acts properly and isometrically on a proper
geodesic metric space $(\mathrm X, d)$. The group $G$ is assumed to be \textit{non-elementary}: $G$ is not virtually cyclic.
We fix a basepoint $o\in X$, and then we consider the induced left-invariant pseudo-metric $d_G$ on group $G$ as follows: $d_G(g_1,g_2)=d(g_1o,g_2o), \forall g_1,g_2\in G$. Denote the
ball of radius $n$ by
$$B(n):=\{g\in G: d_G(1, g)\le n\}.$$
A notable result by Gromov \cite{Gromov81} states that virtually nilpotent groups are characterized by the polynomial growth of $\sharp B(n)$ in Cayley graphs, whereas most groups have exponential growth. The \textit{growth rate} of a subset $A\subset G$ is defined as
\begin{equation}\label{criticalexpo}
\omega(A) = \limsup\limits_{n \to \infty} \frac{\log \sharp(B(n)\cap A)}{n},
\end{equation}
which is independent of the choice of basepoint $o \in \mathrm X$.

Recently, Dahmani, Futer, and Wise\cite{DFW19} showed that for a non-elementary hyperbolic group $G$ and an infinite-index quasi-convex subgroup $H$, their growth rates have strictly gap $\omega(H)<\omega(G)$. They also estabilished the same gap results for relatively hyperbolic groups and infinite-index relatively quasi-convex subgroups, as well as for a group acts geometrically on a CAT(0) cube complex and every subgroup stabilising an essential hyperplane. 
By Legaspi \cite{Legaspi24}, the growth rate gap also exists for a proper group action with a strongly contracting element and an infinite-index quasi-convex subgroup, as well as for an hierarchically hyperbolic group with a Morse element and an infinite-index Morse subgroup.
For a group $G$ with the Morse local-to-global property and an infinite-index stable subgroup $H$, under the assumptions that $G$ is virtually torsion-free or that $H$ is residually finite, Cordes, Russell, Spriano, and Zalloum\cite{CRSZ22} proved that the growth rates also have a gap $\omega(H)<\omega(G)$. Moreover, they proposed an open problem \cite[Problem $5$]{CRSZ22} of whether the torsion-free or residually finite assumptions can be removed. The current paper aims to provide a positive answer to their question by employing a distinct approach rather than the automatic structures for Morse geodesic words.
%to prove growth rate results with no additional assumptions. The virtually torsion-free or residually finite assumptions can be removed. These answer the Problem $5$ in their paper \cite{CRSZ22}.
As usual, a group $G$ is called acts \textit{geometrically} on $X$ if $G$ acts properly and cocompactly on $X$.

\begin{Theorem}\label{main thm}
Suppose that a non-elementary group $G$ acts geometrically on a geodesic metric space $(X,d)$ with the Morse local-to-global property. Assume that $H<G$ is a stable subgroup of infinite index. If $X$ has non-empty Morse boundary, then the growth rate has a gap:
$$\omega(G)>\omega(H).$$
\end{Theorem}

One can easily observe that stable subgroups are always hyperbolic, so for an infinite stable subgroup, there always exists a Morse element, and hence the Morse boundary is non-empty. In this regard, we add the condition that $X$ has non-empty Morse boundary just to avoid some trivial cases.
%In this regard, the condition that $X$ has non-empty Morse boundary seems to be redundant.

Russell, Spriano, and Tran \cite{RST22} obtained the following various groups or spaces admit the Morse local-to-global property, and hence they have the growth rate gap.

\begin{Corollary}
Suppose that $G$ is one of the following non-elementary group acting on its Cayley graph.
\begin{enumerate}
\item The mapping class group of an oriented finite type surface.
\item A graph product of hyperbolic groups.
\item The fundamental group of any closed 3-manifold without Nil or Sol components.
\end{enumerate}
Or suppose that $G$ acts geometrically on a CAT(0) space $(X,d)$ with a Morse element. 
Assume that $H<G$ is a stable subgroup of infinite index. Then the growth rate has a gap:
$$\omega(G)>\omega(H).$$
\end{Corollary}

In mapping class group, the stable subgroups are precisely the convex cocompact subgroups by Durham
and Taylor \cite{DT15}. While in right-angled Artin groups, the stable subgroups are just the purely loxodromic subgroups by Koberda, Mangahas, and Taylor \cite{KMT17}. For a group acts geometrically on a CAT(0) cube complex, our result implies that there is a growth rate gap with respect to an infinite-index stable subgroup, and such a gap also exists for an hyperplane stabilizer subgroup, hence Cordes, Russell, Spriano, and Zalloum \cite[Problem $7$]{CRSZ22} asked when an hyperplane stabilizer is a stable subgroup.

\paragraph{\textbf{The structure of this paper.}}
In Section \ref{Section2}, we give some notations and definitions. After giving a quick review of the Morse boundary, we recall some basic results of Poincar\'{e} series. In Section \ref{Section3}, we list several useful facts about the Morse boundary and show that there exists a Morse element whose fixed points is not contained in the limit set of the stable subgroup. In Section \ref{Section4}, we further obtain that there is a subset of the stable subgroup whose free product with a Morse element admits an embedding into the ambient group and complete the proof of Theorem \ref{main thm}.

\paragraph{\textbf{Acknowledgements.}}
The authors would like to thank Wen-Yuan Yang for fruitful discussion and helpful conversations. Han is supported by NSFC of China(grant No. 12401082). Liu is partially supported by NSFC (grant No.12301084) and Natural Science Foundation of Tianjin(grant No. 22JCQNJC01080).

\section{Preliminary}\label{Section2}

\subsection{Notations and conventions}
Let $(X, d)$ be a geodesic metric space. If all closed ball in $X$ is compact, we say $X$ is a proper geodesic space. The $r$-neighborhood of a subset $A\subset X$ is denoted by $\mathcal{N}_r(A)$.

\begin{Definition}[Hausdorff distance]
Let $X$ be a metric space. The Hausdorff distance between two subsets $A,B\subset X $ is defined by
$$\inf\{r| A\subset \mathcal{N}_r(B), B\subset \mathcal{N}_r(A)\}.$$
\end{Definition}

\begin{Definition}[quasi-isometry; quasi-geodesic]
A map $f: (X,d_X)\to (Y, d_Y)$ between metric spaces is called a $(K, \epsilon)$-quasi-isometric embedding if $K\geq 1,\epsilon\geq 0$, and for each $x, x'\in X$
$$\frac{1}{K}d_X(x,x')-\epsilon\leq d_Y(f(x), f(x'))\leq Kd_X(x, x')+\epsilon.$$
We say $f$ is a quasi-isometry if there exists a constant $C\geq 0$ such that $d_Y(y, f(X))\le C$ for every $y\in Y$. 
If $X$ is a segment of $\mathbb{R}$, then we call $f$ is a $(K, \epsilon)$-quasi-geodesic.
\end{Definition}

Let $(X, d)$ be a proper geodesic metric space. Given a point $x\in X$ and a closed subset $A\subset X$, the set of points $a\in A$ such that $d(x,a)=d(x,A)$ is denoted by $\pi_A(x)$. The \textit{projection} of a subset $Y\subset X$ to $A$ is $\pi_A(Y)\coloneqq\cup_{y\in Y}\pi_A(y)$.

For any $x,y\in X$, a choice of geodesic between them is denoted by $[x, y]$.
In Section \ref{Section4} we usually consider a rectifiable path $\alpha\subset X$ with arc-length parametrization, its length is denoted by $\ell(\alpha)$. Let $\alpha_-,\alpha_+$ be its initial and terminal points respectively. 

To simplify the calculation, in the subsequent discussion we will use the following equivalent definition of quasi-geodesics: a path $\alpha$ is called a $(\lambda,c)$-quasi-geodesic for $\lambda\geq 1, c\geq 0$ if $\ell(\beta)\leq\lambda d(\beta_-,\beta_+)+c$ for any rectifiable subpath $\beta\subseteq\alpha$.
%Let $I$ denote a closed interval of $\mathbb{R}$. Then a map $\alpha: I\rightarrow X$ is a $(\lambda,c)$-quasi-geodesic if $\alpha$ is a $(\lambda,c)$-quasi-isometric embedding:
%\[\frac{1}{\lambda}|s-t|-c\leq d(\alpha(s),\alpha(t))\leq\lambda|s-t|+c,~\forall s,t\in I.\]
%We usually use the image of $\alpha$ to denote the quasi-geodesic for convenience.

A quasi-geodesic $\alpha\subset X$ is called \textit{$M$-Morse} for a function $M:[1,+\infty)\times[0,+\infty)\rightarrow[0,+\infty)$ if any $(\lambda,c)$-quasi-geodesic with endpoints $p,q\in\alpha$
%$\alpha(t_1),\alpha(t_2)$ $(t_1<t_2)$
is contained in the $M(\lambda,c)$-neighborhood of the subpath $[p,q]_\alpha$ of $\alpha$.
%$[\alpha(t_1),\alpha(t_2)]_\alpha:=\alpha|_{[t_1,t_2]}$ of $\alpha$.
If a $(\lambda,c)$-quasi-geodesic $\alpha$ is $M$-Morse, then we call $\alpha$ is a \textit{$(M;\lambda,c)$-Morse quasi-geodesic}.
A path $\alpha\subseteq X$ is a \textit{$(L;M;\lambda,c)$-local Morse quasi-geodesic} if for any subpath
$[p,q]_\alpha\subseteq\alpha$ with length $\ell([p,q]_\alpha)\leq L$, we have
$[p,q]_\alpha$ is a $(M;\lambda,c)$-Morse quasi-geodesic.
The function $M$ is usually called the \textit{Morse gauge} of $\alpha$.

\begin{Definition}
We call a metric space $(X,d)$ is a \textit{Morse local-to-global space} if for any Morse gauge $M:[1,+\infty)\times[0,+\infty)\rightarrow[0,+\infty)$ and any constants $\lambda\geq1,c\geq0$,
there exists a Morse gauge $M'$ and constants $L\geq0,\lambda'\geq1,c'\geq0$ so that every $(L;M;\lambda,c)$-local Morse quasi-geodesic in $X$ is a global $(M';\lambda',c')$-Morse quasi-geodesic.
A group $G$ is called a \textit{Morse local-to-global group} if some (or equivalently, any) of its Cayley graph is a Morse local-to-global space.
\end{Definition}

The Morse local-to-global property of a metric space is invariant under quasi-isometry.

\begin{Definition}
Suppose that a group $G$ acts properly on $(X,d)$. A subgroup $H\leq G$ is \textit{stable} if for some basepoint $o\in X$, there exists Morse gauge $M$ and $D\geq0$, so that for any $h\in H$, the geodesic $[o,ho]$ is $M$-Morse which is contained in the $D$-neighborhood of $Ho$. An element $g\in G$ is \textit{Morse} if the subgroup $\langle g\rangle$ is stable.
\end{Definition}

For a hyperbolic space, the Gromov boundary is the collection of geodesic rays up to the equivalence that two rays have finite Hausdorff distance. The \textit{Morse boundary} $\partial_M X$ of a proper geodesic metric space $(X,d)$ is defined similarly as follows: for a fixed basepoint $o\in X$, it is the set of equivalent classes of all Morse geodesic rays based at $o$, where two rays are equivalent if their Hausdorff distance is finite.

Given a Morse gauge $N$, we consider the following subset of $\partial_M X$ which consists of all equivalent classes of geodesic rays with Morse gauge $N$:
\[\partial_M^N X=\{[\alpha]: \text{there exists an $N$-Morse geodesic ray $\beta\in[\alpha]$ with base point $o$}\}.\]
Here $[\alpha]$ is the equivalence class of the geodesic ray $\alpha$.
We topologize this set with the compact-open topology. 

Let $\mathcal{M}$ be the set of all Morse gauges. We put a natural partial ordering on $\mathcal{M}$ and define the Morse boundary of $X$ to be
$$\partial_M X=\displaystyle\lim_{\displaystyle\overrightarrow{\mathcal{M}}} \partial_M^N X$$
with the induced direct limit topology, that is, a set $V$ is open in $\partial_M X$ if and only if $V\cap \partial_M^NX$ is open for every Morse gauge $N$.

For more details and proofs, the reader can check them in \cite{Cor17}. The Morse boundary is basepoint independent and moreover it is a quasi-isometry invariant.

Suppose that a group $G$ acts properly on the proper geodesic metric space $(X,d)$ and $H\leq G$ is a subgroup of $G$. Cordes and Durham \cite[Definition 3.2]{CD19} introduced the limit set of the action $H$ on $X$, denoted by $\Lambda_MH$, as following 
\[\Lambda_MH=\{\xi\in \partial_M X|\exists \{h_n\}\subseteq H,\lim h_n=\xi\} \]
%\begin{align*}
%\Lambda_MH=\{\xi\in \partial_M X| &\exists ~\text{Morse gauge}~ N ~\text{and a sequence}~ \{h_n\}\subseteq H, \\ & \text{so that}~ [o, h_no] ~\text{is}~ N-\text{Morse and}~ [o, h_no]\rightarrow\xi\}
%\end{align*}
where we say a sequence $\{g_n\}\subseteq G$ converges to a point $\xi\in \partial_M X$, denoted by $\lim g_n=\xi$, if for a fixed basepoint $o\in X$, there exists a Morse gauge $N$, so that $[o, g_no]$ is $N$-Morse for each $n$ and the sequence $\{[o, g_no]\}$ converges uniformly on compact subsets to a geodesic ray $\alpha\in \partial_M^N X$ representing $\xi$. It is easy to verify that the convergence does not depend on the choice of basepoint.

For a Morse element $g\in G$, its limit set refers to the limit set of $\langle g\rangle$, which is exactly two point $\{g^+,g^-\}=:Fix_{\partial_MX}(g)$ with $\lim\limits_{n\rightarrow+\infty} g^n=g^+,\lim\limits_{n\rightarrow-\infty} g^n=g^-$. The Morse element acts on the Morse boundary with a weak north-south dynamics \cite[Corollary 6.9]{Liu21}.

\subsection{Growth gap criterion}
%{A critical gap criterion}

Suppose that a group $G$ acts properly on a metric space $(X, d)$. For a basepoint $o\in X$, recall that $d_G$ is the proper pseudo-metric on $G$ induced by the orbit $Go$.
For a subset $A\subseteq G$, the corresponding Poincar\'e series
$$\mathcal{P}_{A}(s)= \sum\limits_{g\in A} \exp(-s\cdot d_G(1,g)), \; s \ge 0$$
is clearly divergent for $s<\omega(A)$ and convergent for $s>\omega(A)$.
For the case $s=\omega(A)$,
Dal'Bo, Peign\'{e}, Picaud, and Sambusetti \cite{DPPS11} presented a useful approach to obtain the divergence of the Poincar\'e series in the critical exponent.

\begin{Lemma} \label{lem:Divergence}
Suppose that a group $G$ acts properly on a proper geodesic metric space $(X,d)$. Then for a quasi-convex subgroup $H\leq G$, the Poincar\'{e} series $\mathcal{P}_{H,d_G}(s)$ is divergent at $\omega_{d_G}(H)$.
\end{Lemma}

\begin{proof}
We first divide the orbit $Ho$ into a collection of annulus sets
\[\mathcal{A}(n,\Delta)=\{h\in H:|d_G(1,h)-n|\leq\Delta\},\]
for some $\Delta\geq0$.
Then we apply the same argument as in the proof of \cite[Proposition 4.1 (1)]{DPPS11}, where we observe that the quasi-convexity of $H$ is enough to obtain that
$\mathcal{A}(n+m,\Delta)\subseteq\mathcal{A}(n,\Delta)\cdot\mathcal{A}(m,\Delta)$
for an appropriate $\Delta\geq0$, and hence
\[\sharp\mathcal{A}(n+m,\Delta)\leq\sharp\mathcal{A}(n,\Delta)\cdot\sharp\mathcal{A}(m,\Delta).\]
Noticing further that the Poincar\'e series can be re-formed as
\[\mathcal{P}_{A}(s)=\sum\limits_{r\geq0}e^{-sr} \sharp\mathcal{A}(r,0)\leq\sum\limits_{r\geq0}e^{-sr} \sharp\mathcal{A}(r,\Delta).\]
Then by a result for a positive squence with submultiplicative property\cite{PS7276}, one can obtain that $\sharp\mathcal{A}(n,\Delta)\geq e^{\omega(H)n}$, and then an easy estimation tells us that $\mathcal{P}_{H,d_G}(s)$ is divergent at $\omega_{d_G}(H)$.
\end{proof}

For two subsets $A, B\subseteq G$, let $\mathbb{W}(A, B)=\{a_1b_1\cdots a_nb_n: a_i\in A,b_i\in B,0<n\in\mathbb{N}\}$ be the collection of all words with letters alternating in $A$ and $B$.
The following result by Yang is essentially used to estabilish the growth rate gap.

\begin{Lemma}\cite[Lemma 2.23]{Yang19}\label{degrowth}
Suppose that a group $G$ acts properly on $(X,d)$.
Assume that $A, B\subseteq G$ are two subsets so that the evaluation map $\iota:\mathbb{W}(A, B)\to G$ is injective. If $B$ is finite, then $\mathcal{P}_{A,d_G}(s)$ converges at $s=\omega(X)$ where $X:=\iota(\mathbb{W}(A,B))\subset G$ is the image.
%In particular,   we have the critical gap:$$\omega_{d_G}(X) > \omega_{d_G}(A)$$ provided that $\mathcal{P}{A, d_G}(s)$ is divergent at $s=\omega_{d_G}(A)$.
\end{Lemma}

%\begin{Remark}
%From now on, we shall always consider the metric $d_G$ as the pullback of the metric $d$ on $\mathrm Y$ via the proper action. Hence, the subindex $d_G$ is omitted for simplicity.
%\end{Remark}

\section{Properties of the Morse boundary}\label{Section3}

Throughout the paper, we assume that $(X,d)$ is a proper geodesic metric space with non-empty Morse boundary, and suppose that a group $G$ acts geometrically on $X$. We always fix a basepoint $o\in X$.
By the well-known Svarc-Milnor Lemma, the space $(X,d)$ is quasi-isometric to the Cayley graph of $G$. If a property is invariant under quasi-isometry, then we can say the group action admits such a property provided that the Cayley graph of the ambient group does so.

In a Morse local-to-global space, a result in \cite{RST22} tells us the existence of a Morse element in a group. Morse property and Morse local-to-global property are invariants under quasi-isometry.

\begin{Lemma}\label{lem:ExistMorseEle}\cite[Proposition 4.8 ]{RST22}
Let $(X,d)$ be a Morse local-to-global space with non-empty Morse boundary. If $G$ acts geometrically on $X$, then $G$ contains a Morse element.
\end{Lemma}

A finitely generated group $G$ is called non-elementary if it is not virtually cyclic. The second author showed that if a non-elementary group has non-empty Morse boundary, then its Morse boundary is an infinite set \cite[Corollary 5.8]{Liu21} and the action on the Morse boundary is minimal.

\begin{Lemma} \label{lem:MinimalAction}\cite[Theorem 6.1, Corollary 6.3]{Liu21} 
Suppose that a non-elementary group $G$ acts geometrically on the space $(X,d)$ with non-empty Morse boundary. Then $G$ acts minimally on the Morse boundary $\partial_MX$. That is, for any $p\in \partial_MX$, the orbits $G\cdot\{p\}$ is dense in $\partial_MX$.

Moreover, if $G$ contains a Morse element, 
then the set of rational points $$\{g^+|\ g \mbox{ is a Morse element in } G\}$$ is dense in $\partial_MX$.
\end{Lemma}

From the two lemmas above, one can get the following corollary easily, which was proved in \cite[Theorem 6.5]{CRSZ22} by regular languages.
\begin{Corollary}
Let $G$ be a Morse local-to-global group with non-empty Morse boundary. Then the set $\{g^+|\ g \mbox{ is a Morse element in } G\}$ is dense in $\partial_*G$.
\end{Corollary}

The following lemma tells us the connection between fixed points of the Morse element and the limit set $\Lambda_M H$.
\begin{Lemma}\label{g1}
Let $(X,d)$ be a Morse local-to-global space with non-empty Morse boundary and $G$ acts geometrically on $X$. Let $H$ be an infinite index stable subgroup of $G$ and $g\in G$ be a Morse element. If $Fix_{\partial_M X}(g)\cap\Lambda_M H\neq \emptyset$, then $Fix_{\partial_M X}(g)\subset \Lambda_M H$.
\end{Lemma}

%\begin{Lemma}\label{g1}
%Let $G$ be a Morse local-to-global group with non-empty Morse boundary and $g$ be a Morse element of $G$. Let $H$ be an infinite index stable subgroup of $G$. If $Fix_{\partial_{*}G}(g)\cap\Lambda H\neq \emptyset$, then $Fix_{\partial_{*}G}(g)\subset \Lambda H$.
%\end{Lemma}

\begin{proof}
Without loss of generality, suppose that $g^+\in \Lambda_M H$. Then by the stable property of $H$ and $\langle g\rangle$, we can find a constant $C>0$ such that, there exists a sequence $\{h_n\}\subseteq H$ and $k_n\rightarrow+\infty$ satisfying $d(h_no, g^{k_n}o)<C$.
By the proper action of $G$, the ball $B(C)=\{g\in G:d(o,go)\leq C\}$ contains only finite elements of $G$. Hence there exist two distinct integers $m, n$ such that $h_{n}^{-1}g^{k_n}=h_{m}^{-1}g^{k_m}$.
Then it is easy to see that $h_mh_{n}^{-1}$ and $g$ have the same limit set on the Morse boundary $\partial_M X$. This implies $Fix_{\partial_M X}(g)\subset \Lambda_M H$.
\end{proof}

Cordes and Durham\cite{CD19} obtained that the hyperbolicity of a group can be characterized by the compactness of its Morse boundary. This is also showed in the second author paper\cite{Liu21}

\begin{Lemma} \label{lem:HyperbolicCompact} \cite[Corollary 1.17]{CD19}\cite[Theorem 5.7]{Liu21}
Suppose that a group $G$ acts geometrically on the space $(X,d)$ with non-empty Morse boundary. Then $X$ is hyperbolic if and only if the Morse boundary $\partial_M X$ is compact.
\end{Lemma}

The following observation tells us that one can always find a Morse element such that its fixed points miss the limit set $\Lambda_M H$.
This lemma will be used in the next section to construct the embedding free product of subsets. 

\begin{Lemma}\label{MorseExists}
Suppose that $(X,d)$ is a Morse local-to-global space with non-empty Morse boundary and $G$ acts geometrically on $X$. Then for any infinite index stable subgroup $H$ of $G$, there exists a Morse element $g\in G$ such that $Fix_{\partial_M X}(g)\cap \Lambda_M H=\emptyset$.
\end{Lemma}

%\begin{Lemma}\label{MorseExists}
%Let $G$ be a Morse local-to-global group with non-empty Morse boundary. For any finitely generated set $S$ of $G$ and any infinite index stable subgroup $H$ of $G$. Then there exists a Morse element $g\in G$ such that $Fix_{\partial_M X}(g)\cap \Lambda_M H=\emptyset$.
%\end{Lemma}

\begin{proof}
We will show it by contradiction. First by Lemma \ref{lem:ExistMorseEle}, we know that $G$ contains Morse elements. Now suppose that $Fix_{\partial_M X}(g)\cap\Lambda_M H$ is a nonempty set for any Morse element $g\in G$. By Lemma \ref{g1}, we have $Fix_{\partial_M X}(g)=\{g^-, g^+\}\subset \Lambda_M H$ for all Morse element $g \in G$. Note that $\Lambda_M H$ is a closed subset in the Morse boundary $\partial_M X$. By Lemma \ref{lem:MinimalAction}, the action of $G$ on the Morse boundary $\partial_M X$ is minimal, which implies that, for any point $p\in\partial_M X$, the orbits $G\cdot\{p\}$ is dense in $\partial_M X$.
Hence in particular, for a fixed Morse element $f$, we have 
$$G\cdot(f^+)\subset \{g^+, g^-|\ g \mbox{ is a Morse element in } G\}\subset \Lambda_M H\subset \Lambda_M G=\partial_M X.$$
By taking the closure, this shows that $\Lambda_M H=\Lambda_M G$. Since $H$ is a stable subgroup of $G$, all $[o,ho]$ for $h\in H$ are uniformly Morse, hence $H$ is a hyperbolic group, and thus its Morse boundary $\Lambda_M H$, which is also its Gromov boundary, is compact. This means that the Morse boundary $\Lambda_M G$ of $G$ is compact. By Lemma \ref{lem:HyperbolicCompact}, we have that the group $G$ must be hyperbolic.

However, for a hyperbolic group $G$, if $H\leq G$ is stable, then $H$ is quasi-convex in $G$. Moreover, if their limit set coincides with each other, i.e., $\Lambda_M H=\Lambda_M G$, then $H$ acts cocompactly on the weak convex hull 
%(that is, the union of quasi-geodesics ending at the limit set) 
of the limit set of $H$ which is co-bounded in $G$. Thus $H$ is finite index in $G$. We get a contraction.

%However, for any stable subgroup $H$ of a hyperbolic group $G$ with the same limit set as $G$, we always have $[G : H]<\infty$. We get a contraction.
%On the other hand, we know that a stable subgroup $H$ in a hyperbolic group must be quasi-convex. If its limit set coincides with the Gromov boundary of $G$, then $H$ acts cocompactly on the weak convex hull (i.e.: the union of quasi-geodesics ending at the limit set)  of the limit set of $H$ which is co-bounded in the Cayley graph $G$. Thus  $H$ is of finite index.
\end{proof}

\section{Embedding free product of subsets} \label{Section4}

This section is devoted to prove the main theorem of the present paper. In the process,  we actually construct an embedding of a free product of a subset of the stable subgroup with a Morse element into the ambient Morse local-to-global group, and this will imply the result by the usual method on divergent Poincar\'{e} series.

A subset $A\subseteq G$ is called \textit{$R$-separated} if $d(a, a')\geq R$ for any two $a\ne a'\in A$. The following lemma is common in the literature, cf. \cite[Lemma 3.6]{PY19}.

\begin{Lemma}\label{lem:SepNet}
Suppose that a group $G$ acts properly on $(X,d)$. Then for any $R>0$, there exists $\theta=\theta(R)>0$ with the following property. Assume that $A\subseteq G$ is a subset, then there exists an $R$-separated
subset $A'\subseteq A$ so that $A\subseteq N_{R}(A')$ and $\sharp (A'\cap B(n))\geq \theta\cdot\sharp (A\cap B(n))$ for any $n\in\mathbb{R}$. In particular, $\omega(A')=\omega(A)$.
\end{Lemma}

\begin{proof}
Consider the collection of all $R$-separated subsets in $A$ with the partial order by inclusion. Then we can choose a maximal $A'\subseteq A$ in this collection by the axiom of choice. Hence we have $A\subseteq N_R(A')$. By the proper action of $G$ on $(X,d)$, there exists $N\in\mathbb{N}$ so that there are at most $N$ elements of $Go$ belonging to a ball of radius $R$. Take $\theta=\frac{1}{N}$, and the result follows.
\end{proof}

In the following, We always assume that $(X,d)$ is a Morse local-to-global proper metric space with non-empty Morse boundary, and assume that a group $G$ acts properly on $(X,d)$.
Suppose that $g\in G$ is a Morse element, then we denote by $A(g)=\langle g\rangle o$ the quasi-axis of $g$.

\begin{Lemma} \label{lem:BoundedProjection}
Assume that $H<G$ is a stable subgroup of infinite index, and assume that $g\in G$ is a Morse element with $Fix_{\partial_{M}G}(g)\cap \Lambda H=\emptyset$. Then there exists $C_0>0$ so that
\[\mathrm{diam} \pi_{A(g)}(Ho)\leq C_0, \mathrm{diam} \pi_{Ho}(A(g))\leq C_0.\]
\end{Lemma}

\begin{proof}
Suppose otherwise that $\mathrm{diam} \pi_{A(g)}(Ho)=\infty$, then noticing that $g^jo\in\pi_{A(g)}(Ho)$ if and only if $g^{-j}o\in\pi_{A(g)}(Ho)$, we can find a sequence $k_i\rightarrow+\infty$, so that $g^{k_i}o\in\pi_{A(g)}(h_io)$ for some $h_i\in H$. Since $\langle g\rangle$ is stable, there exists $D'\geq0$ so that $[o,g^{k_i}o]$ is contained in the $D'$-neighborhood of $\langle g\rangle o=A(g)$ for each $i$.
We claim that the concatenated paths $\gamma_i=[o,g^{k_i}o][g^{k_i}o,h_io]$ for $i\geq 1$ are $(3,2D')$-quasi-geodesics. To see this, we only need to consider the subpath $[p_i,q_i]_{\gamma_i}$ with $p_i\in[o,g^{k_i}o]$ and $q_i\in[g^{k_i}o,h_io]$.
Let $p_i'\in\pi_{[o,g^{k_i}o]}(q_i)$ be a shortest projection point of $q_i$ to $[o,g^{k_i}o]$. Then $d(p_i',g^{l_i}o)\leq D'$ for some $g^{l_i}\in\langle g\rangle$ by the stable property, and we have $d(q_i,g^{k_i}o)\leq d(q_i,g^{l_i}o)$ by noticing that $g^{k_i}o$ is a shortest projection point of $q_i$ to $A(g)$. Combing $d(q_i,p_i')\geq d(q_i,g^{l_i}o)-d(p_i',g^{l_i}o)\geq d(q_i,g^{k_i}o)-D'$, we have
\begin{align*}
\ell([p_i,q_i]_{\gamma_i})&=d(p_i,g^{k_i}o)+d(g^{k_i}o,q_i)\leq d(p_i,q_i)+d(q_i,g^{k_i}o)+d(p'_i,q_i)+D' \\
& \leq d(p_i,q_i)+2d(p'_i,q_i)+2D'\leq3d(p_i,q_i)+2D'.
\end{align*}
%\vspace{-1em}
\begin{figure}[htbp]
  \vspace{-2em}
  \centering
  \includegraphics[width=0.3\textwidth]{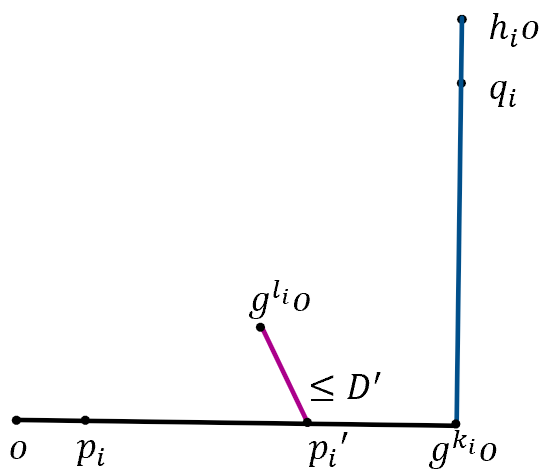}
  %\caption{Subdiagram}\label{figure:subdiagram}
  \vspace{-0.5em}
\end{figure}

Since $H$ is stable, there exists Morse gauge $M$ and $D\geq0$ so that $[o,h_io]$ is $M$-Morse which is contained in the $D$-neighborhood of $Ho$, and hence $\gamma_i$ is contained in the $D'':=(M(3,2D')+D)$-neighborhood of $Ho$.
In particular, $[o,g^{k_i}o]\subseteq N_{D''}(Ho)$ for each $i$, which implies that $A(g)\subseteq N_{D'''}(Ho)$ for some $D'''\geq0$, contradicting the assumption $Fix_{\partial_{M}G}(g)\cap \Lambda H=\emptyset$.

In the same way, we can obtain that $\mathrm{diam} \pi_{Ho}(A(g))$ is also bounded.
\end{proof}

\begin{Corollary}\label{cor:LineProjBod}
Under the same assumptions as Lemma \ref{lem:BoundedProjection}, the following projections are also uniformly bounded,
for any $h\in H$ and $k\in\mathbb{Z}$,
\[\mathrm{diam} \pi_{[o,g^ko]}([o,ho])\leq C_1, \mathrm{diam} \pi_{[o,ho]}([o,g^ko])\leq C_1.\]
\end{Corollary}

\begin{proof}
Given $h\in H$ and $k\in\mathbb{Z}$, we take $b\in[o,ho]$, then we choose a $a\in\pi_{[o,g^ko]}(b)$. Notice that $a$ is a shortest projection point, then a same argument as Lemma \ref{lem:BoundedProjection} we obtain that $[o,a][a,b]$ is a $(3,0)$-quasi-geodesic. By the stable property of $H$, we have that $[o,a][a,b][b,h'o]$ is contained in the $M(3,4D)$-neighborhood of $Ho$ for some $h'\in H$ with $d(b,h'o)\leq D$, where $M$ is the Morse gauge and $D\geq0$.
Moreover, since $\langle g\rangle$ is stable, there exists $g^l\in\langle g\rangle$ with $d(g^lo,a)\leq D'$ for some $D'\geq0$. Therefore, $d(g^lo,Ho)\leq M(3,4D)+D'$, then a similar discussion as Lemma \ref{lem:BoundedProjection} completes the proof of the part $\mathrm{diam} \pi_{[o,g^ko]}([o,ho])\leq C_1$. The other part can be obtained in the same way.
\begin{figure}[htbp]
  \vspace{-1em}
  \centering
  \includegraphics[width=0.4\textwidth]{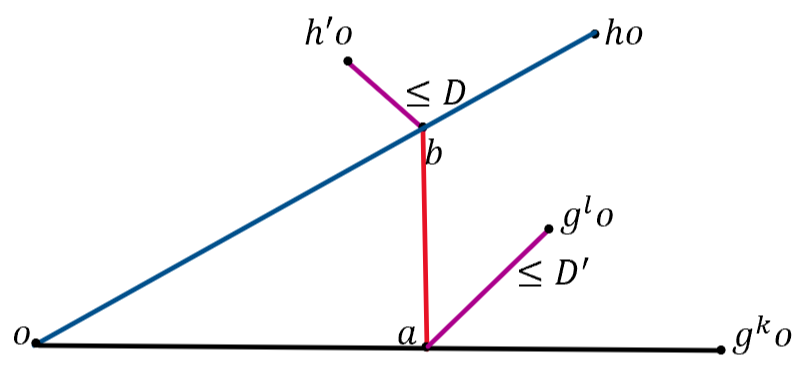}
  %\caption{Subdiagram}\label{figure:subdiagram}
  \vspace{-1em}
\end{figure}
%We first observe that the union $\cup_{h\in H}[o,ho]$ is contained in a $D$-neighborhood of $Ho$ for some $D\geq0$ by the stable property of $H$. Hence their projections to a set $A\subseteq X$ always have
%Suppose otherwise that there exists a sequence $k_i\rightarrow+\infty$ and a sequence $\{h_i\}\subseteq H$, so that $\lim\limits_{i\rightarrow+\infty}\mathrm{diam} \pi_{[o,g^{k_i}o]}(h_io)=+\infty$. We can take $a_i\in\pi_{[o,g^{k_i}o]}(h_io)$ with $\lim\limits_{i\rightarrow+\infty}d(o,a_i)=+\infty$. By the stable property of $\langle g\rangle$, all $[o,g^{k_i}o]$ are contained in a $D'$-neighborhood of $A(g)$ for some $D'\geq0$. Therefore, there exists $g^{l_i}\in\langle g\rangle$ so that $d(a_i,g^{l_i}o)\leq D'$ for each $i$.
\end{proof}

\begin{Lemma} \label{lem:Injective}
Suppose that $(X,d)$ is a Morse local-to-global metric space with non-empty Morse boundary, and suppose that a group $G$ acts properly on $(X,d)$. Assume that $H<G$ is a stable subgroup of infinite index.
Then there exists a Morse element $g\in G$ and
a constant $C\geq0$, as well as a subset $H'\subseteq H$, such that $H\subseteq N_C(H'),\omega(H')=\omega(H)$ and the evaluation map $\iota:\mathbb{W}(H',\{g\})\to G$ is injective.
\end{Lemma}

\begin{proof}
By Lemma \ref{MorseExists}, there exists a Morse element $g$ such that $Fix(g)\cap \Lambda H=\emptyset$.
Let $M_1,M_2$ be the Morse gauges for $H$ and $\langle g\rangle$ respectively by the stable property. And let $M_0=\max\{M_1,M_2\}$.
We define a function
\begin{align*}
M:[1,+\infty)\times[0,+\infty)&\rightarrow[0,+\infty) \\
(\lambda,c)&\mapsto M_0(2\lambda+1,c)=\max\{M_1(2\lambda+1,c),M_2(2\lambda+1,c)\}
\end{align*}
\paragraph{\textbf{Claim:}}
For any $h\in H$ and $k\in\mathbb{Z}$, any subpath $\alpha$ of $[ho,o][o,g^ko]$ or $[g^ko,o][o,ho]$ is a $(M;3,2C_1)$-Morse quasi-geodesic, where $C_1$ is given by Corollary \ref{cor:LineProjBod}.

\begin{proof}[Proof of the claim]
Suppose that $\alpha\subseteq[ho,o][o,g^ko]$, the other case can be verified similarly.
We only need to consider that $o\in\alpha$. By taking a shortest projection point, we get that $[ho,o][o,g^ko]$ and thus $\alpha$ are $(3,2C_1)$-quasi-geodesic in a similar way as Lemma \ref{lem:BoundedProjection}.
Assume that $\beta$ is a $(\lambda,c)$-quasi-geodesic with endponits $\beta_-=p\in\alpha\cap[ho,o]$ and $\beta_+=q\in\alpha\cap[o,g^ko]$. Take a shortest projection point $a\in\pi_{\beta}(ho)$, then
\begin{align*}
\ell([ho,a][a,q]_{\beta})&=d(ho,a)+\ell([a,q]_{\beta})\leq d(ho,a)+\lambda d(a,q)+c \\
&\leq d(ho,a)+\lambda[d(a,ho)+d(ho,q)]+c\leq(2\lambda+1)d(ho,q)+c.
\end{align*}
Moreover, a same argument as above we obtain that $[ho,a][a,q]_{\beta}$ and $[p,a]_{\beta}[a,ho]$ are $(2\lambda+1,c)$-quasi-geodesics, and hence they are contained in a $M_0(2\lambda+1,c)$-neighborhood of $[ho,q]_{\alpha}$ and $[p,ho]_{\alpha}$ respectively. As a result, $\alpha$ is $M$-Morse, and the claim completes.
\end{proof}
Now for the Morse gauge $M$ and constants $3,2C_1$, the Morse local-to-global property tells us that there exists a Morse gauge $M'$ and constants $L\geq0$, $\lambda'\geq1$, $c'\geq0$ so that every $(L;M;3,2C_1)$-local Morse quasi-geodesic in $X$ is a global $(M';\lambda',c')$-Morse quasi-geodesic.

Take a large $k\in\mathbb{N}$ so that $d(o,g^ko)\geq \max\{L,c'+1\}$, and we also write $g$ for $g^k$ for easy notation. Then we take an $L$-separated subset $H'\subseteq H$ so that $H\subseteq N_L(H')$ by Lemma \ref{lem:SepNet}.

Given a word $W=h_1gh_2g\cdots h_ng\in\mathbb{W}(H',\{g\})$ with $h_i\in H'$, by the claim we have that the path labelled by $W$
\[[o,h_1o]\cdot h_1([o,go])\cdot h_1g([o,h_2o])\cdot\cdots\cdot h_1gh_2g\cdots h_{n-1}gh_n([o,go])=:\gamma\]
%\[[o,h_1o]\cdot h_1([o,go])\cdot h_1g([o,h_2o])\cdot\cdots\cdot h_1gh_2g\cdots h_{n-1}g([o,h_no])\cdot h_1gh_2g\cdots h_{n-1}gh_n([o,go])=:\gamma\]
is a $(L;M;3,2C_1)$-local Morse quasi-geodesic, and hence a global $(M';\lambda',c')$-Morse quasi-geodesic.
So $d(o,\iota(W)o)\geq\frac{1}{\lambda'}(\ell(\gamma)-c')\geq\frac{1}{\lambda'}(d(o,go)-c')\geq\frac{1}{\lambda'}>0$, which implies that $W$ is nontrivial in $G$. Therefor, the map $\iota:\mathbb{W}(H',\{g\})\to G$ is injective.
\end{proof}

\begin{proof}[Proof of Theorem \ref{main thm}]
Suppose that $H\leq G$ is a stable subgroup. Then $H$ is obviously quasi-convex in $G$, and by Lemma \ref{lem:Divergence}, the Poincar\'{e} series $\mathcal{P}_{H}(s)$ is divergent at $\omega(H)$.
By Lemma \ref{lem:Injective}, there exists a Morse element $g\in G$ and $H'\subseteq H$ so that $\omega(H')=\omega(H)$ and the evaluation map $\iota:\mathbb{W}(H',\{g\})\to G$ is injective. And by lemma \ref{degrowth}, the Poincar\'{e} series $\mathcal{P}_{H'}(s)$ is convergent at $\omega(\iota(\mathbb{W}(H',\{g\})))=:\omega_0$ and hence at $\omega(G)\geq\omega_0$. Noticing that the two series $\mathcal{P}_{H}(s)$ and $\mathcal{P}_{H'}(s)$ are controlled by each other in a common point $s$, $\mathcal{P}_{H'}(s)$ is divergent at $\omega(H)$, thus we have that $\omega(H)<\omega(G)$.
\end{proof}

\bibliographystyle{plain}
\bibliography{Ref}

\bigskip
{School of Mathematics, Hunan University, Changsha, Hunan, 410082, P.R.China}

{\tt Email: hansz@hnu.edu.cn}

\bigskip
{School of Mathematical Sciences \& LPMC, Nankai University, Tianjin 300071, P.R.China}

{\tt Email: qingliu@nankai.edu.cn}

\end{document}